\documentclass[12pt]{amsart}

\usepackage{times}      % use Times fonts for text
\usepackage{amssymb}    % use all AMS fonts
\usepackage{amsmath}    % include some AMS-LaTeX functionality
\usepackage{amsthm}
\usepackage{graphicx}

\usepackage{url}
\makeatletter
\def\url@leostyle{%
  \@ifundefined{selectfont}{\def\UrlFont{\sf}}{\def\UrlFont{\scriptsize\ttfamily}}}
\makeatother
\newtheorem{theorem}{Theorem}[section]

\newtheorem{slicethm}[theorem]{Slice Theorem}
\newtheorem{soulthm}[theorem]{Soul Theorem}
\newtheorem{GBonnetMyers}[theorem]{Generalized Bonnet-Myers Theorem}
\newtheorem{GSTheorem}[theorem]{Generalized Synge's Theorem}
\newtheorem{GSlemma}[theorem]{Generalized Synge's Lemma}
\newtheorem{Stability}[theorem]{Stability Theorem}
\newtheorem*{thma}{Theorem A}
\newtheorem*{thmb}{Theorem B}
\newtheorem*{thmc}{Theorem C}
\newtheorem*{thmd}{Theorem D}
\newtheorem*{CorE}{Corollary E}
\newtheorem*{MSR}{Maximal Symmetry Rank Theorem}
\newtheorem*{fphthm}{Fixed-Point Homogeneity Theorem}
\newtheorem*{cht}{Covering Homotopy Theorem}
\newtheorem{corollary}[theorem]{Corollary} 

\newtheorem{lemma}[theorem]{Lemma} 
\newtheorem{proposition}[theorem]{Proposition} 
\newtheorem{example}[theorem]{Example}

\def\symrank{\textrm{symrank}}

\def\rrr{\mathbb{R}}
\def\ccc{\mathbb{C}}
\def\zzz{\mathbb{Z}}
\def\hh{\mathbb{H}}

\def\blowX{X^{!}_p}
\def\GblowX{X^{!}_{G(p)}}

\DeclareMathOperator{\Isom}{Isom}
\DeclareMathOperator{\codim}{codim}
\DeclareMathOperator{\Fix}{Fix}

\def\bdm{\begin{displaymath}}
\def\edm{\end{displaymath}}
\def\normal{\vartriangleleft}
\def\dist{\textrm{dist}}
\def\gexp{\textrm{gexp}}

\def\rk{\textrm{rk}}
\def\txt{\textrm}
\def\d{\partial}
\def\dim{\textrm{dim}}
\def\cl{\textrm{cl}}

\def\ra{\rightarrow}

\DeclareMathOperator{\Cohomfix}{cohomfix}
\DeclareMathOperator{\curv}{curv}

\theoremstyle{definition}
\newtheorem{definition}[theorem]{Definition} 

\newtheorem*{acknowledge}{Acknowledgements}

\numberwithin{equation}{section}
\begin{document}
%%%%%%%%%%%%%%%%%%%%%%%%%%%%%%%%%%%%%%%%%%%%%%%%%%%%%%%%%%%%%
\newcommand{\comment}[1]{\vspace{5 mm}\par \noindent
\marginpar{\textsc{Note}}
\framebox{\begin{minipage}[c]{0.95 \textwidth}
#1 \end{minipage}}\vspace{5 mm}\par}

\title[Orientation and symmetries of Alexandrov spaces]{Orientation and symmetries of Alexandrov spaces with applications in positive curvature}

\author[Harvey]{John Harvey}
\address{Department of Mathematics, University of Notre Dame, Notre Dame, Ind. 46556, U.S.A.}
\email{jharvey2@nd.edu}

\author[Searle]{Catherine Searle}
\address{Oregon State University, Department of Mathematics, 368 Kidder Hall, Corvallis, Oregon, 97331, U.S.A.}
\email{searle.catherine@gmail.com}


\thanks{The first-named author was supported in part by a grant from the U.S. National Science Foundation. The second-named author was supported in part by 
CONACYT Project \#SEP-CO1-46274 and CONACYT Project \#SEP-82471.}

\date{\today}

\begin{abstract}
We develop two new tools for use in Alexandrov geometry: a theory of ramified orientable double covers and  a particularly useful version of the Slice Theorem for actions of compact Lie groups. These tools are applied to the classification of compact, positively curved Alexandrov spaces with maximal symmetry rank.
\end{abstract}
\maketitle

\section*{Introduction}
In the study of geometry, the Riemannian manifold has long been a subject of intense investigation.
The smooth structure allows for many attractive and useful tools. However, 
geometry is 
essentially the study of distances and angles on a space, and so the metric space is clearly the natural home of the subject.  
In order to understand the impact of curvature bounds we must restrict our attention further: to so-called CAT($\kappa$) spaces for upper curvature bounds and to Alexandrov spaces for lower bounds.

By choosing to study Alexandrov spaces one foregoes techniques which the Riemannian geometer takes for granted, such as the existence of convexity and injectivity radii, extendibility of geodesics and isotopies via vector fields, but in return the class of spaces is much richer. The creation of new tools with which to work in these spaces can involve a lot of technical work. We present here two adaptations of completely fundamental concepts from Riemannian geometry: the theory of orientable double covers and the 
Slice Theorem.

The applications of this paper show that once the necessary tools have been made available, they can be used to develop simple and appealing proofs. In this paper, we provide proofs of the counterparts of two classification theorems from Riemannian geometry which are perhaps simpler than the original proofs (admittedly, this is, in part, because we cannot aim for a diffeomorphism classification) and which cast a new light on the old results. 

It is well known that every Riemannian manifold $M$ has an orientable double cover $\tilde{M}$, and that there is a free orientation-reversing isometric involution $i:\tilde{M} \ra \tilde{M}$ such that $M$ is isometric to $\tilde{M} / i$.
One of the useful aspects of the class of Alexandrov spaces is that it is closed under taking quotients by isometric group actions even when those actions are not free. Therefore it is natural to 
consider involutions that fix points. We obtain the following result.

\begin{thma}Let $X$ be an Alexandrov space of dimension $n$ and $\curv \geq k$ which is non-orientable. Then there is an orientable Alexandrov space $\tilde{X}_{\txt{Ram}}$ with the same dimension and lower curvature bound, and with an isometric involution $i$  such that $\tilde{X}_{\txt{Ram}} / i$
and
$X$ are isometric. $\tilde{X}_{\txt{Ram}}$ is a ramified orientable double cover of $X$, and the ramification locus is the union of those strata in $X$ having non-orientable normal cones.\end{thma}

When studying the action of a compact Lie group on a topological space, the Slice Theorem is a crucial component of the theory. It is clear that the theorem is true in Alexandrov geometry, simply because Alexandrov spaces are completely regular topological spaces \cite{MY57}. However, in Riemannian geometry we can go further, identifying the slice with the normal space to the orbit. We show that an analogous identification is possible in Alexandrov geometry. 

\begin{thmb} Let $G$, a compact Lie group, act isometrically on an Alexandrov space $X$. Then for all $p \in X$, there is some $r_0 > 0$ such that for all $r < r_0$ there is an equivariant homeomorphism $\Phi : G \times_{G_p} K\nu_p \rightarrow B_r(G(p))$, where $K\nu_p$ is the cone on the space of normal directions to the orbit $G(p)$.\end{thmb}

 Recall that the group of isometries of an Alexandrov space is a Lie group \cite{FY}, just as for Riemannian manifolds \cite{MySt}.
 The project of classifying positively curved Riemannian manifolds with ``large" isometry groups, where the  
size of the group action may be 
 interpreted in a variety of ways, can therefore reasonably be extended to positively curved Alexandrov spaces.  
The present work shows that 
this extension provides more information about the nature of the different curvature-symmetry conditions.
 
The maximum dimension for the isometry group of an Alexandrov space is the same as
in the Riemannian case, and when this dimension is achieved the Alexandrov space is homogeneous \cite{GGG}, and so a Riemannian manifold \cite{Be}. In contrast to this, non-manifold Alexandrov spaces of cohomogeneity one \cite{GGS} exist in all dimensions greater than or equal to $3$. In particular, the suspension of any positively curved homogeneous space is a positively curved space of cohomogeneity one.

One notion of largeness is when the fixed point set of the group action, $X^G$, has maximal possible dimension in the space, that is when $\dim(X^G) = \dim(X/G) - 1$. Such actions are called {\it fixed-point homogeneous}, and the positively curved Riemannian manifolds of positive sectional curvature admitting such actions were classified in \cite{GS2}. 

\begin{fphthm} Let $G$, a compact Lie group, act isometrically and fixed-point homogeneously on $M^n$, a closed, simply-connected, positively curved Riemannian
manifold, with $M^G \neq \emptyset$. Then $M^n$ is diffeomorphic to one of $S^n$, $\ccc P^k$, $\hh P^m$ or ${\rm Ca}P^2$, where $2k=4m=n$.
\end{fphthm}

The orbit spaces of fixed-point homogeneous actions on positively curved Riemannian manifolds have a structure which certainly invites study, but at first sight one might well be surprised to find that only the rank one symmetric spaces admit such actions. Broadening the question to positively curved Alexandrov spaces sheds some light on why this is the case.

\begin{thmc} 
Let $G$, a compact Lie group, act isometrically and fixed-point homogeneously on $X^n$, a compact $n$-dimensional Alexandrov space of positive curvature,  
with $X^G\neq \emptyset$.
If $F$ is the component of $X^G$ with maximal dimension, then the following hold:
\begin{itemize}
\item[(i)] There is a unique orbit $G(p)\cong G/G_p$ at maximal distance from $F$ (the ``soul" orbit).
\item[(ii)]  The space $X$ is $G$-equivariantly homeomorphic to  
$(\nu *G)/G_p$, where $\nu$ is the space of normal directions to the orbit at $p$.
\end{itemize}
\end{thmc}

We see from this that fixed-point homogeneous spaces are plentiful among the positively curved Alexandrov spaces. In fact, the two sets have the same cardinality, since for every positively curved space $\nu$, its join to a positively curved homogeneous $G$-space yields a fixed-point homogeneous $G$-space. Comparison of the two results suggests that fixed-point homogeneity is only a highly restrictive measure of symmetry in the context of Riemannian manifolds. Indeed, the proof of the original theorem relies completely on $\nu$ being isometric to the round sphere.

Another possible measure of symmetry  is 
the {\it symmetry rank} of the space, where 
\bdm
\symrank(X)= \rk(\Isom(X)).
\edm
 Closed Riemannian manifolds with positive curvature and maximal symmetry rank were classified in \cite{GS}. 

\begin{MSR}\label{t:msrriemannian} Let $M$ be an $n$-dimensional, closed, connected Riemannian manifold with positive sectional curvature.
Then 
\begin{itemize}
\item[(1)]
$\symrank(M)\leq \lfloor \frac{n+1}{2}\rfloor$.
\vskip .1cm
\item[(2)]  Moreover,  equality holds in $(1)$ only if $M$ is diffeomorphic to a sphere, a real or complex projective space 
or a lens space.
\end{itemize}
\end{MSR}

The corresponding result for Alexandrov spaces provides an alternative statement of the result, and gives a somewhat different view of the phenomenon.

\begin{thmd}\label{t:MSRA} Let $X$ be an $n$-dimensional, compact,
 Alexandrov space with $\curv \geq 1$ admitting an isometric, (almost) effective $T^k$-action. Then $k\leq \lfloor \frac{n+1}{2} \rfloor$ and in the case of equality either
\begin{enumerate}
\item $X$ is a spherical orbifold, homeomorphic to $S^n / H$, where $H$ is a finite subgroup of the centralizer of the maximal torus in $O(n+1)$ or
\item only in the case that $n$ is even, $X \cong S^{n+1}/H$, where $H$ is a rank one subgroup of the maximal torus in $O(n+2)$.
\end{enumerate}
In both cases the action on $X$ is equivalent to the induced action of the maximal torus on the $G$-quotient of the corresponding sphere.
\end{thmd}

We can now see precisely why the list of maximal symmetry rank manifolds is short: it is because the orthogonal group contains very few subgroups which simultaneously commute with the maximal torus and act freely on the sphere. 
In fact, in $O(2n)$, $Z(T^n)=T^n$ and the only subgroups of $T^n$ that can act freely on $S^{2n-1}$ are finite cyclic or the diagonal circle. In contrast, in $O(2n+1)$, no subgroup of $T^n$ acts freely, but $Z(T^n)\neq T^n$ and the antipodal map is the only subgroup of $Z(T^n)$ that does act freely on $S^{2n}$.

Indeed, it is clear from the proof of the Riemannian version of the theorem that the group actions are always induced by the maximal torus in the orthogonal group (cf. also  \cite{GG2}, or \cite{MY}, noting that the latter contains a more general result applicable to the spherical space-forms). 

Alexandrov spaces of maximal symmetry rank must be topological manifolds in dimensions less than or equal to $3$, whereas in dimension $4$, there are numerous examples of Alexandrov spaces of maximal symmetry rank that are not manifolds, and in dimension $6$ there are examples which are not even orbifolds. However, since these spaces are all ultimately spherical in origin, we can conclude
that the combination of positive curvature with maximal symmetry rank is restrictive in its own right, and not only in the Riemannian setting.

An immediate corollary of Theorem D is the classification of positively curved Riemannian orbifolds of maximal symmetry rank.

\begin{CorE}Let $X$ be an $n$-dimensional, closed
 Riemannian orbifold with positive sectional curvature admitting an isometric, (almost) effective $T^k$-action. Then $k\leq \lfloor \frac{n+1}{2} \rfloor$ and in the case of equality either
\begin{enumerate}
\item $X$ is a spherical orbifold, homeomorphic to $S^n / \Gamma$, where $\Gamma$ is a finite subgroup of the centralizer of the maximal torus in $O(n+1)$ or
\item only in the case that $n$ is even, $X$ is homeomorphic to a finite quotient of a weighted complex projective space $\ccc P^n_{a_0,a_1,\cdots,a_n}/\Gamma$, where $\Gamma$ is a finite subgroup of the linearly acting torus.
\end{enumerate}
\end{CorE}

The paper is organized as follows. In Section \ref{s:prelim}, we will recall some general facts about Alexandrov spaces. In Section \ref{s:lno}, we develop the theory of ramified orientable double covers, proving Theorem A. In Section \ref{s:slice}, we will consider isometric group actions on Alexandrov spaces and prove the Slice Theorem, and in Section \ref{s:fixedpoints}, we generalize other well-known results on Riemannian group actions. In Section \ref{s:fph}, we prove Theorem C, which is then applied in Section \ref{s:msr} to prove Theorem D.

Finally, we note that 
in a forthcoming paper \cite{HS} we will prove the following result for positively curved spaces having almost maximal symmetry rank in dimension 4. These spaces are also quotients of spheres, but there is a larger list of admissible groups.

\begin{theorem} Let $T^1$ act isometrically and effectively on $X^4$, where $X$ is a compact, positively curved, orientable Alexandrov space. Then up to homeomorphism
$X^4$ is any orientable suspension (in which case it is given by $S^4/H$ where $H$ is a finite subgroup of the centralizer of a $T^1 \subset SO(5)$) or is $S^5/H$ where $H$ is a rank one subgroup of the centralizer of any $T^2 \subset SO(6)$.
\end{theorem}

In the special case of positively curved, simply-connected Alexandrov $4$-manifolds, Galaz-Garc\'ia has 
shown that such spaces are equivariantly homeomorphic to $S^4$ or $\ccc P^2$ \cite{GG}.

\begin{acknowledge}
The authors are grateful to 
Christine Escher, Karsten Grove, Alexander Lytchak, Ricardo Mendes, Anton Petrunin, and Conrad Plaut for helpful conversations, as well as to Fernando Galaz-Garc\'ia for initial conversations with C.\ Searle from which this paper evolved. The work in this article forms a part of J. Harvey's doctoral dissertation, which is being completed with the insightful advice of Karsten Grove.  C. Searle is grateful to the Mathematics Department of the University of Notre Dame for its hospitality during a visit where a part of this research was carried out. 
\end{acknowledge}

\section{Preliminaries}\label{s:prelim}

In this section we will first fix notation and then recall basic definitions and theorems about Alexandrov spaces.
%NOTATION

We will denote an Alexandrov space by $X$, and will always assume it is complete and finite-dimensional. Given an isometric (left) action $G\times X\rightarrow X$ of a Lie group $G$, and a point $p\in X$, we let $G(p)=\{\,gp :g\in G \,\}$ be the \emph{orbit}  of $p$ under the action of $G$. The \emph{isotropy group} of $p$ is the subgroup $G_p=\{\, g\in G: gp=p\,\}$. Recall that $G(p)\cong G/G_p$. We will denote the orbit space of this action by $\bar{X}=X/G$. Similarly, the image of a point $p\in X$ under the orbit projection map $\pi:X\rightarrow \bar{X}$ will be denoted by $\bar{p}\in \bar{X}$.

We will assume throughout that $G$ is compact and
in Section 5 that its action is either \emph{effective} or \emph{almost effective}, i.e., that $\bigcap_{p\in X}G_p$ is respectively either trivial or a finite subgroup of $G$. We will call two $G$-actions \emph{equivalent} when there is a $G$-equivariant homeomorphism between the two $G$-spaces.

We will always consider the empty set to have dimension $-1$. Homeomorphisms and isomorphisms will be represented by $\cong$, while isometries will be represented by $=$. We will use $T^1$ to refer to the circle as a Lie group, and $S^1$ to refer to it as a topological space without any group structure. 

\subsection{Basics of Alexandrov geometry}\label{ss:basics} 

A finite-dimensional \textit{Alexandrov space} is a locally complete, locally compact, connected (except in dimension 0, where a two-point space is admitted by convention) length space, with a lower curvature bound in the triangle-comparison sense. Like most authors, we will assume that the space is complete. For non-complete spaces, we will follow \cite{Pet1} in using the term \textit{Alexandrov domain}. Every point in an Alexandrov domain has a closed neighborhood which is an Alexandrov space \cite{PP1}. There are a number of introductions to Alexandrov spaces to which the reader may refer for basic information (cf.\ \cite{ BBI, BGP, P,  Pl, Sh}).

A more analytic formulation of the curvature condition, using the concavity of distance functions, was introduced in \cite{PP2}. We say that a function $f: \rrr \ra \rrr$ is \emph{$\lambda$-concave} if it satisfies the differential inequality $f'' \leq \lambda$, in the barrier sense. A function $f : X \ra \rrr$ on a length space $X$ is $\lambda$-concave if its restriction to every geodesic (shortest path) is $\lambda$-concave. We use the term \emph{semi-concave} to describe functions which are locally $\lambda$-concave, where $\lambda$ need not have a uniform upper bound. Semi-concave functions on Alexandrov spaces have a well-defined gradient, and in particular, this gives rise to a gradient flow (introduced in \cite{PP2}, see also \cite{Pet3}).

We
say that a complete length space $X$ is an Alexandrov space with $\curv(X)\geq k$ if, for any point $p$,
a modification of $\dist(p,\cdot)$
satisfies a certain concavity condition.
In particular, if we let $f=\rho_k \circ \dist(p,\cdot)$, where 

\begin{equation}\label{e:1} 
\rho_k(x) = \begin{cases} 1/k(1-\cos(x\sqrt{k})), & \mbox{if } k>0 \\
x^2/2, & \mbox{if } k=0  \\
1/k(1-\cosh(x\sqrt{-k})), & \mbox{if } k<0, \\ \end{cases}
\end{equation}
then $f$ must be $(1-kf)$-concave. Note that in a space form of constant curvature $k$ equality holds, that is, $f'' = 1-kf$.

The \emph{space of directions} of an Alexandrov space $X^n$ of dimension $n$ at a point $p$ is,
by definition, the completion of the 
space of geodesic directions at $p$ and is denoted by $\Sigma_p X$ or, if there is no confusion, $\Sigma_p$.  For any subset $Y$ of $X^n$, we denote the space of directions tangent to $Y$ at $p \in Y$ by $\Sigma_p Y$.
The space of directions of $X^n$ is a compact Alexandrov space of dimension $n-1$ with $\curv \geq 1$.
We recall here a particularly useful result for such spaces (cf.\ \cite{GWannals}), which can be seen as a natural counterpart to the Splitting Theorem in non-negative curvature.

\begin{lemma}[Join Lemma]\label{l:join} Let $X$ be an $n$-dimensional Alexandrov space with  $\curv \geq 1$. If $X$ contains an isometric copy of the unit round sphere $S^m_1$, then $X$ is isometric to the spherical join $S_1^m*\nu$, where $\nu$ is an isometrically embedded $(n-m-1)$-dimensional Alexandrov space with $\curv \geq 1$ which we will refer to as the \emph{normal space} to $S^m_1$ .
\end{lemma}

The cone on $\Sigma_p X$, endowed with the Euclidean cone metric, is 
called the \emph{tangent cone} and is written as $T_p X$. In fact, 
$(\frac{1}{r}X , p)\ra (T_p X, o)$, as $r \ra 0$.

\subsection{Singularities and Topology}

In order to understand Alexandrov spaces, 
a grasp of their local structure is required. This was achieved with Perelman's Stability Theorem \cite{P} (cf.\ \cite{K}).

\begin{Stability}\label{t:stability}
Let $X^n_i$ be a sequence of Alexandrov spaces with curvature uniformly bounded from below, converging to an Alexandrov space $X^n$ of the same dimension and let $\lim_{i \to \infty} o(i) = 0$. Let $\theta_i: X \ra X_i$ be a sequence of $o(i)$-Hausdorff approximations. Then for all large $i$ there exist homeomorphisms $\theta'_i : X \ra X_i$, $o(i)$-close to $\theta_i$.
\end{Stability}

In particular, it follows that a small metric ball centered at $p \in X$ is homeomorphic to $T_p X$, and so  Alexandrov spaces are spaces with multiple conic singularities in the following sense \cite{Si}.

\begin{definition}$X$ is a \emph{space with multiple conic singularities} or \emph{MCS space} of dimension $n$ if and only if every point $p \in X$ has a neighborhood which is pointed homeomorphic to an open cone on a compact MCS space of dimension $n-1$, where the unique MCS space of dimension $-1$ is the empty set.\end{definition}

An important difference between Riemannian manifolds and Alexandrov spaces is the existence of these singularities. We refer to a point $p\in X$ as \emph{regular}   if $\Sigma_p$ is isometric to the unit round sphere and as \emph{singular} otherwise.  We make the further distinction that a point is \emph{topologically singular} if its space of directions is not homeomorphic to a sphere. The set of regular points of an Alexandrov space is dense and convex, while the singular points ``may be arranged chaotically" \cite{PP1}.

By restricting our attention to certain types 
of singularities, we can stratify an Alexandrov space into topological 
manifolds in two different ways: the first stratification is purely topological and the second, which takes into account metric information, is by extremal sets.

 The canonical stratification  of an MCS space into topological manifolds is given as follows: a point $p \in X$ belongs to the $l$-dimensional stratum $X^{(l)}$ if $p$ has a conic neighborhood homeomorphic to $\rrr^l \times K$, where $l$ has been chosen to be maximal and $K$ is a cone on a compact MCS space. We will describe 
$K$ as the normal cone to the stratum $X^{(l)}$. For more information on this topological stratification of Alexandrov spaces, see \cite{P4}. Using this stratification, we can see that the codimension of the set of topologically singular points, other than boundary points, is at least $3$.

The more refined stratification by \textit{extremal sets} 
is given in \cite{PP1}. A non-empty, proper extremal set comprises points with spaces of directions which  differ significantly from the unit round sphere. They can be defined as the sets which are ``ideals" of the gradient flow of $\dist(p,\cdot)$ for  every point $p$. Examples of extremal sets are isolated points with spaces of directions of diameter $\leq \pi/2$, the boundary of an Alexandrov space and, in a trivial sense, the entire Alexandrov space. The restriction of this stratification to the boundary partitions the boundary into \emph{faces}.  We refer the reader to \cite{Pet3} for definitions and important results.

\subsection{Quasigeodesics and Radial Curves}\label{ss:qgrc}

When an extremal set is given its intrinsic path metric, the shortest paths in the extremal set exhibit characteristics similar to geodesics. We can generalize the notion of geodesic to include these ``quasigeodesics" as follows.

\begin{definition}\label{d:1} A curve $\gamma$ in $X$, an Alexandrov space with $\curv \geq k$, is a \textit{quasigeodesic}  if and only if it is parametrized by arc-length and
$f(t) = \rho_k \circ \dist(p, \gamma(t))$ is $(1-kf)$-concave, where $\rho_k$ is as
defined in Equation \ref{e:1}.
\end{definition}

The natural generalization of totally geodesic submanifolds from Riemannian geometry is the totally quasigeodesic  subset.
 We say that a closed subset $Y\subset X$ is totally quasigeodesic if 
a shortest path in $Y$ between points is a quasigeodesic in the ambient space $X$. See \cite{Pet3} for the formal definition. Extremal sets are the most important example of totally quasigeodesic subsets.

In Riemannian geometry, geodesics are used to define an exponential map. While this can be done in Alexandrov geometry, the domain of the exponential map does not in general contain an open neighborhood of the origin of the tangent cone. Instead, we can define a \emph{gradient exponential} map \cite{PP2, Pet3}, which relies on an alternative generalization of geodesics: instead of a geodesic starting at a point with a particular initial direction we use a \emph{radial curve}. 

\begin{definition}\label{d:gexp}
Let $X$ be an Alexandrov space and let $p \in X$. Let $f = \frac{1}{2} \dist(p,\cdot)^2$. Let $i_s: sX \ra X$ denote the canonical dilation map, and let $\Phi^t_f:X\ra X$ denote the gradient flow of $f$ for time $t$. Consider the convergence $(e^t X , p) \ra (T_p X, o)$ as $t \ra \infty$. The \emph{gradient exponential} at $p$, $\gexp_p : T_p X \ra X$ is given by $\lim_{t \ra \infty} \Phi^t_f \circ i_{e^t}$.
\end{definition}

This is the optimal definition for a non-negatively curved space, where $\gexp_p$ is 1-Lipschitz, though it still works for most applications regardless of the particular curvature bound. Note that $f=\rho_0 \circ \dist(p,\cdot)$ is 1-concave in such a space. In the general case, on a bounded ball $B_r(p)$ the concavity of $f$ is still well-controlled and so the maps in the sequence are all locally Lipschitz, and on any bounded set the Lipschitz constant is uniform. Therefore a partial limit does exist. For any partial limit we have $\Phi^t_f \circ \gexp_p(v) = \gexp_p(e^t v)$, and hence the limit map is unique.

\begin{definition}
The \emph{radial curve} starting at $p$ in the direction $v \in \Sigma_p$ is the curve $\gamma: [0,\infty) \ra X$ given by $\gamma(t)=\gexp_p(tv)$.
\end{definition}
  
The radial curves are gradient curves for $\dist(p,\cdot)$ which have been reparametrized in a way which coincides with the arc-length parametrization so long as the radial curve is a geodesic. 

If we wish to obtain more control over the behaviour of the gradient exponential, as in the proof of Theorem A, we can adjust the definition \cite{Pet3}. It is simplest to adjust the definition of the radial curve, and then redefine the gradient exponential accordingly. First note that a standard radial curve $\alpha(t)$ starting at $p$ in the direction $\xi$ is the solution of the differential equation
$$\alpha^+(t) = \frac{|p \alpha(t)|}{t} \nabla \alpha(t) \mbox{ for all } t > 0, \mbox{ with } \alpha(0)=p \mbox{ and } \alpha^+(0)=\xi.$$

We change the differential equation to 
\begin{align*}
\alpha^+(t) &= \frac{\sinh(|p \alpha(t)|)}{\sinh(t)} \nabla \alpha(t) \mbox{ for }k=-1\mbox{ and }\\
\alpha^+(t) &= \frac{\tan(|p \alpha(t)|)}{\tan(t)} \nabla \alpha(t) \mbox{ for }k=1
\end{align*}
$$$$
and rescale appropriately for other curvature bounds.

The corresponding gradient exponential maps will still be 1-Lipschitz if we change the metric on the tangent cone, defining it by the hyperbolic, 
respectively spherical, law of cosines, instead of the Euclidean law.

We can define a partial inverse to the exponential at $p$, which we will refer to as the \emph{logarithm} and denote by $\log_p$.  This function is not necessarily uniquely defined or continuous. The composition $\gexp_p \circ \log_p$ is the identity map on $X$, and it can be shown that the rescaled logarithm $r \cdot \log_p : \frac{1}{r} B_r(p) \ra B_1 (o) \subset T_p X$ is a Gromov-Hausdorff approximation witnessing the convergence of $(\frac{1}{r}X,p) \ra (T_p X,o)$ \cite{BBI}.

\subsection{Classical Theorems for Positive and Non-Negative Curvature}

For a Riemannian manifold of positive sectional curvature there are two important theorems that 
characterize its topology. They are the Bonnet-Myers Theorem, which tells us that if the curvature is bounded away from zero the manifold is compact and the fundamental group is finite, and Synge's Theorem, which tells us that in even dimensions an orientable manifold of positive curvature is simply connected and in odd dimensions a manifold of positive curvature is orientable. In non-negative curvature, the Soul Theorem controls the topology of open manifolds and of compact manifolds with boundary.

As
pointed out in \cite{Pet1}, Alexandrov spaces are unlike manifolds in that they can have arbitrarily small neighborhoods which do not admit an orientation. In particular, if $p \in X$ has a non-orientable space of directions $\Sigma_p$, then no neighborhood of $p$ is orientable. We call such a space \emph{locally non-orientable}.

Petrunin \cite{Pet1} proved an analogue of Synge's Theorem for locally orientable Alexandrov spaces, which we recall here without the hypothesis of local orientability in even dimensions, a simple improvement which follows directly from the results in Section \ref{s:lno}.

\begin{GSTheorem}\label{t:GST} Let $X^n$ be an 
Alexandrov space with $\curv \geq 1$. 
\begin{enumerate}
\item If $X$ is even-dimensional and is either orientable or locally non-orientable
 then $X$ is simply connected, otherwise it has fundamental group $\zzz_2$.
\item If $X$ is odd-dimensional and locally orientable then $X$ is orientable.
\end{enumerate}
\end{GSTheorem}

The analogue of the Bonnet-Myers Theorem for general Alexandrov spaces is well known but could not be located by the authors elsewhere in the literature. For completeness, a proof is presented here.

\begin{GBonnetMyers}\label{t:BonnetMyers} Let $X$ be an 
Alexandrov space of $\curv \geq k >0$. Then 
$X$ is compact and has finite fundamental group.
\end{GBonnetMyers}

\begin{proof} Since an Alexandrov space of $\curv \geq k>0$ has diameter bounded above by $\pi/\sqrt{k}$ \cite{BGP}, it follows from local compactness that $X$ must  be compact. 
Since Alexandrov spaces are 
MCS spaces, they have universal covers (cf.\  \cite{BP}), and it is clear that the metric on $X$ induces a metric on the universal cover with the same lower curvature bound. The proof now proceeds just as in the Riemannian case.
\end{proof}

We also have the following 
important result for non-negatively curved
 Alexandrov spaces, which will be used throughout the text \cite{P}.

\begin{soulthm} \label{t:soulthm} Let $X$ be a compact Alexandrov space of $\curv\geq 0$ and suppose that $\d X\neq \emptyset$.
Then there exists a totally convex, compact subset $S\subset X$, called \emph{the soul} of $X$, with $\d S=\emptyset$, which is a strong deformation retract of $X$. If $\curv(X)>0$, then the soul is a point, that is, $S=\{s\}$.   

\end{soulthm}

The proof relies on the concavity of $\dist(\d X, \cdot)$. The gradient flow of the distance function on $X \setminus \d X$ is $1$-Lipschitz, and so can be extended to all of $X$. This flow plays the role of the Sharafutdinov retraction. Note that instead of using the distance from $\d X$ we may use the distance from any union of boundary faces (see \cite{W}). When $\curv (X)>0$ the boundary has only one component and is homeomorphic to $\Sigma_s$. 

Consider the special case where $X$ is the quotient space of an isometric group action on an Alexandrov space $Y$ with $\curv> 0$. Let $\pi:Y\rightarrow X=Y/G$ be the projection map, with $\d X\neq \emptyset$. We see that   $\d X\cong\Sigma_s\cong \nu_p/G_p$, where $p\in \pi^{-1}(s)$. Because the gradient flow preserves extremal subsets, it is clear that the orbit space has a somewhat conical structure. The non-principal orbit types are either $s$, contained within $\d X$, or stretch from $\d X$ to $s$.

\section{Ramified Orientable Double Covers}\label{s:lno}

Recall that Alexandrov spaces are unlike manifolds in that they can have arbitrarily small neighborhoods which do not admit an orientation. 

\begin{example} The spherical suspension of the projective plane, $\Sigma(\rrr P^2)$, does not admit a local orientation in a neighborhood of the cone points. Note that this space is simply connected, whereas a simply-connected manifold is always orientable.
\end{example}

We will use Alexander-Spanier cohomology with integer coefficients to study orientability, as it has certain advantages in the context of Alexandrov spaces (cf. \cite{GP}), and coincides with singular cohomology.

It is easy to see by excision 
that $H^n (X, X \setminus \{ p \}) \cong H^{n-1} (\Sigma_p)$. 
If $H^{n-1} (\Sigma_p) \cong \zzz$, then a choice of generator of $H^{n-1} (\Sigma_p)$ is a \emph{local orientation} at $p$. We will say that $X$ is \emph{locally orientable} if a local orientation can be chosen at each point $p \in X$.

We will define orientability of a compact Alexandrov space without boundary, $X$, 
in terms of the existence of a fundamental class. 
That is, as in the manifold case, $X$ is orientable if for every $x\in X$,  $H^{n-1}(X, X\setminus \{ p \}) \ra H^n(X)\cong \zzz$ is an isomorphism. For non-compact Alexandrov spaces without boundary we define orientability using cohomology with compact supports. If the space has boundary, we will use relative cohomology.

The main result of this section is that every non-orientable Alexandrov space can be obtained as the quotient of an orientable Alexandrov space by an isometric involution. Unlike in the manifold case, the involution is not required to be free. Locally non-orientable spaces arise where the involution has fixed points. We will develop the theory only for spaces without boundary, but it is easy to adapt the theory for spaces with boundary. It may help to begin by considering the simplest application of the theory.

\begin{example}The suspension on $\rrr P^2$ can be obtained as the quotient of $S^3$ by the suspension of the antipodal map on $S^2$.\end{example}

We will consider the matter of orientability for a more general class of topological spaces, which we refer to as \emph{non-branching MCS spaces}. By this we mean MCS spaces where the top stratum is a connected manifold, and where the neighborhood of a point is a cone on another non-branching MCS space, or on a two-point space.

Our definition of orientability and local orientability for a connected, non-branching MCS space without boundary is the same 
as that for an Alexandrov space without boundary, allowing for $\Sigma_p$ to now represent the compact 
connected MCS space whose cone gives us the conic neighborhood at $p\in X$, and for $T_p X$ to represent the cone itself.

\begin{lemma}\label{l:manifold}
Let $X^n$ be a non-branching MCS space of dimension $n \geq 2$ without boundary. $X$ is orientable if and only if the topological manifold $X^{(n)}$ is orientable.
\end{lemma}

\begin{proof}
The advantage of Alexander-Spanier cohomology is that for any closed subset $A \subset X$, $H^n_c(X, A)\cong H^n_c(X\setminus A)$. 
Observe further that, as with Alexandrov spaces,
the codimension of the singular set $S = X \setminus X^{(n)}$ of a non-branching MCS space is at least 3.
Thus, as in \cite{GP}, the set $S$ is so small that
\bdm
H^n_c(X) \cong H^n_c(X, S)\cong H^n_c(X\setminus S)\cong  H^n_c(X^{(n)}).
\edm 

Given this, one implication is trivial, and it remains to show that if $X^{(n)}$ is an orientable manifold then $X$ is orientable. In order to do so, we must show that the generator of $H^n_c(X)$, which we identify with the fundamental class of $X^{(n)}$, induces local orientations at each point. We will do this by induction on the dimension of $X$.

The anchor of the induction is in dimension two, where the result holds trivially. We now assume the result is true for for dimensions $\leq n-1$ and 
will show the result holds for dimension $n$.
For any point $p$ in the orientable manifold $X^{(n)}$, the generator of $H^n_c(X^{(n)})$ maps to a generator of $H^n_c(X^{(n)}, X^{(n)} \setminus \{p\})$ as usual, giving a local orientation at $p$. If $p \in S$, then let $U$ be an open conic neighborhood of $p$, so that $U \cong T_p X$. The set $V = U \setminus S$ is then the cone on $\Sigma_p X \setminus \Sigma_p S$. The orientation of $X^{(n)}$ induces an orientation on the $(n-1)$-dimensional manifold $\Sigma_p X \setminus \Sigma_p S$, and by the inductive hypothesis on the space $\Sigma_p X$.
\end{proof}

\begin{theorem}\label{thm:Cover} Let $X$ be a non-branching MCS space of dimension $n$ without boundary which is non-orientable. 
Then there is an orientable MCS space $\tilde{X}_{\txt{Ram}}$ with the same dimension and with an involution 
$i$  such that $\tilde{X}_{\txt{Ram}} / i \cong X$. $\tilde{X}_{\txt{Ram}}$ 
is a ramified double cover of $X$, and the ramification locus is the union of those strata with non-orientable normal 
cones.\end{theorem}

\begin{proof}
Let $X^{(n)}$ be the manifold part of $X$. By Lemma \ref{l:manifold} it is non-orientable, so it has 
an orientable double cover $\tilde{X}^{(n)}$. There is a natural projection $\pi$ from the disjoint union $(X \setminus X^{(n)}) \cup \tilde{X}^{(n)}$ to $X$, namely the sum of the identity map on $X \setminus X^{(n)}$ with the projection $\tilde{X}^{(n)} \ra X^{(n)}$. If we give the disjoint union the pull-back topology induced by $\pi$, then it is a ramified double cover with the required properties.
\end{proof}

In order to apply this technique to Alexandrov spaces we will need to add metric information. This can be done in a manner very similar to the analysis of branched coverings of spaces over an extremal knot carried out in Lemma 5.2 of \cite{GW2}.
A more general and detailed discussion of when the completion of an Alexandrov domain is an Alexandrov space is available in \cite{Petrunin}, or in the forthcoming work of Nan Li\footnote{In his presentation \emph{Volume rigidity on length spaces} at the Conference on Metric Geometry and Applications i.m. Jianguo Cao, at the University of Notre Dame in March 2013, Li announced a certain ``almost convexity" condition on Alexandrov domains which is equivalent to the completion being an Alexandrov space.}.

\begin{thma}Let $X$ be a non-orientable Alexandrov space without boundary of dimension $n$ and $\curv \geq k$. Then there is an orientable Alexandrov space $\tilde{X}_{\txt{Ram}}$ with the same dimension and lower curvature bound, and with an isometric involution $i$  such that $\tilde{X}_{\txt{Ram}} / i$
and $X$ are isometric. $\tilde{X}_{\txt{Ram}}$ is a ramified orientable double cover of $X$, and the ramification locus is the union of those strata in $X$ having non-orientable normal cones.\end{thma}

\begin{proof} The metric result is all that we need to show. When $X$ is compact and $k=0$, the proof is essentially the same as the proof given in  \cite{GW2}. However, in the case where $X$ is non-compact, or $k \neq 0$, we need to make slight alterations to the argument.

We first give an outline of their proof and then briefly comment on the changes required to prove the result in the above cases.

The proof for $X$ compact with $k=0$ is obtained in three steps. In the first step, they show that for any hinge in $\tilde{X}_{\txt{Ram}}$ based at a point $x\in L$, where $L$ is the ramification locus, Toponogov's Theorem holds. This allows for the second step, which is to show that the {\it shadow} of $L$ 
has measure zero, that is, for $q\in \tilde{X}_{\txt{Ram}}\setminus L$,  the set $$S_q=\{s\in \tilde{X}_{\txt{Ram}} : {\textrm{ every minimal geodesic }} qs {\textrm{ from }} q {\textrm{ passes through }} L\}$$ has measure zero.   The third and final step is to show that the distance function from $q$ satisfies the correct concavity condition. It is sufficient to do this on $\tilde{X}_{\txt{Ram}} \setminus S_q$, because $S_q$ has measure zero.

Now, if we consider the case where $k\neq 0$, we note that we must make the following changes to the proofs of each step.
In Step 1 of their argument 
we use the more refined versions of the gradient exponential map from \cite{Pet3} described in Section \ref{ss:qgrc}.

For Step 2, in order to show that the measure of the shadow is zero, we note that 
the estimates will change
according to the spherical and hyperbolic analogues of the law of cosines. Further,  if $\tilde{X}_{\txt{Ram}}$ is not compact, we can exhaust it by a countable union of balls centered at $q$.

In Step 3, we use different concavity conditions, depending on the type of lower curvature bound as in \cite{PP2} (see Section \ref{ss:basics}, Equation \ref{e:1}).
\end{proof}

We will refer to $\tilde{X}_{\txt{Ram}}$ as the \emph{ramified orientable double cover} of $X$.
We now present a lemma on lifting group actions to these ramified orientable double covers, and we will also use the covers to classify positively curved spaces in dimension 3.

\begin{lemma}\label{l:lift}Let $G$ be a connected Lie group acting by isometries on an $n$-dimensional non-orientable Alexandrov space $X$ without boundary. Let $\tilde{X}_{\txt{Ram}}$ be the ramified orientable double cover of $X$. Then the action of $G$ on $X$ lifts to an action of a $2$-fold covering group of $G$, $\tilde{G}$ on $\tilde{X}_{\txt{Ram}}$.\end{lemma}

\begin{proof}Since the $G$-action is isometric, it must preserve the stratification of $X$, and in particular act on the manifold $X^{(n)}$. The action may then be lifted to an action of a 2-fold covering group on the double cover of $X^{(n)}$, which is dense in $\tilde{X}_{\txt{Ram}}$. Since the action is by isometries it extends to all of $\tilde{X}_{\txt{Ram}}$.\end{proof}

The following classification results are an easy consequence of the theory of ramified orientable double covers. They have been obtained independently in \cite{GGG2}, and
also appear to  be known to others in the field \cite{PMO}.

\begin{proposition}\label{t:3dimclassification} The only closed, simply-connected Alexandrov spaces of positive curvature in dimension three are homeomorphic to either $S^3$ or $\Sigma \mathbb{R}P^2$.\end{proposition}

\begin{proof}If $X$ is a manifold, this follows from the resolution of the Poincar\'e Conjecture \cite{P1,P2,P3}. If $X$ is not a manifold, its ramified orientable double cover is $S^3$ \cite{Kwun, NR}, and the only isometric involution on $S^3$ which fixes only isolated points yields $\Sigma \rrr P^2$ \cite{Sm,Li}.\end{proof}

Because $\Sigma \mathbb{R}P^2$ does not admit deck transformations, we have the following corollary.

\begin{corollary} Let $X^3$ be a closed, three-dimensional, Alexandrov space of positive curvature which is not a manifold. Then 
$X^3$ is homeomorphic to $\Sigma \rrr P^2$.
\end{corollary}

\section{Isotropy and the Slice Theorem}\label{s:slice}

The purpose of this section is to show that the isotropy action of a group on the space of normal directions at a point $p$ is equivalent to its action on any slice at $p$. 
The principal obstacle to showing this is in clarifying the relationship between the infinitesimal structure, that is, the space of directions, and the local structure of an Alexandrov space.

We will begin by constructing an analogue of the normal exponential map for a group orbit. We will next develop 
a more precise relationship between the local and infinitesimal structure via the Stability Theorem. Then we will introduce two ``blow-up" constructions, which combine information about the infinitesimal and local structures in a single topological space.

In the context of these blow-up spaces, the connection between an arbitrary slice and the normal space becomes clear, and  allows us to prove the Slice Theorem.

\subsection{Normal Bundles of Group Orbits} 

The following
proposition from \cite{GGS} uses the Join Lemma \ref{l:join} to describe the tangent and normal spaces to an orbit of an isometric group action.

\begin{proposition}\label{p:normal_space} Let $X$ be an Alexandrov space admitting an isometric action of a compact Lie group $G$ and fix $p\in X$ with $\dim(G/G_p)>0$. If $S_p\subset \Sigma_p$ is the unit tangent space to the orbit $G(p)\cong G/G_p$,  then the following hold. 
\begin{itemize}
  \item[(1)] $S_p$ is isometric to the unit round sphere.
  \item[(2)] The set $\nu(S_p)$ is a compact, totally geodesic Alexandrov subspace of $\Sigma_pX$ with 
  $\curv \geq 1$,  and the space of directions $\Sigma_p X$ is isometric to the join $S_p * \nu(S_p)$ with the standard join metric.
  \item[(3)] Either $\nu(S_p)$ is connected or it contains exactly two points at distance $\pi$.
\end{itemize} 
\end{proposition}

From now on, we will simply write $\nu_p$ for $\nu(S_p)$, and $K \nu_p$ for the cone on $\nu_p$. The union of all  normal spaces $\nu_q$ for $q \in G(p)$ is topologized in an obvious way by the action of $G$, giving a normal bundle to the orbit $G(p)$, $G \times_{G_p} K \nu_p$. Just as the exponential map is generalized to the gradient exponential, we now wish to generalize the normal exponential map.

\begin{definition}[\cite{AKP}, compare Definition \ref{d:gexp}]
Let $X$ be an Alexandrov space and let $A \subset X$ be a closed subset, and let $p \in A$. Let $f = \frac{1}{2} \dist(A,\cdot)^2$. Let $i_s: sX \ra X$ denote the canonical dilation map, and let $\Phi^t_f$ denote the gradient flow of $f$ for time $t$. Consider the convergence $(e^t X , p) \ra (T_p X, o)$ as $t \ra \infty$. The \emph{$A$-gradient exponential} at $p$, $\gexp^A_p : T_p X \ra X$ is given by $\lim_{t \ra \infty} \Phi^t_f \circ i_{e^t}$, provided this limit exists.
\end{definition}

If we compare this definition to Definition \ref{d:gexp} and consider the arguments (originally from \cite{Pet3}) given there for the existence of a limit, we see that there is little change. We have the same control on the Lipschitz constants and hence a partial limit always exists. However, the equality $\Phi^t_f \circ \gexp_p(v) = \gexp_p(e^t v)$ can fail, and so the uniqueness of the limit is not guaranteed. 

In \cite{AKP}, the situation where $p \in A$ is an isolated point of $A$ is considered. For points $q$ near $p$ we have $|qA| = |qp|$. The equality 
holds and the map $\gexp^A_p: T_p X \ra X$ is defined near the origin of the tangent cone.

However, this is not the only situation in which such a gradient exponential map can be defined. Let $\pi:  X \ra Y$ be a submetry. For $p \in X$, let $A = \pi^{-1}(\pi(x))$ be the fiber through $p$ and let $H_p$ be the set of points horizontal to $p$, defined by $q \in H_p$ precisely when $|qp| = |qA|$. By \cite{Lyt}, as $\frac{1}{r} (X,p) \ra (T_p X, o)$ we have $\frac{1}{r} (H_p,p) \ra K \nu_p$, where $\nu_p$ is the space of horizontal directions at $p$. It follows that the map $\gexp^A_p$ exists if its domain  is restricted to the horizontal cone $K \nu_p$.

Now, in the particular case where $G$ acts on $X$, we set $A = G(p)$ and define a gradient exponential on $K \nu_p$ so that every horizontal ray maps to an $A$-radial curve starting at $p$. This extends to a map defined on the entire normal bundle to $G(p)$, $\gexp_{G(p)} : G \times_{G_p} K \nu_p \ra X$, which we call the \emph{normal gradient exponential} map.

\subsection{Stability Theorem}

 The Stability Theorem \ref{t:stability} guarantees the existence of homeomorphisms from $T_p X$ to small metric neighborhoods of $p$ which will serve as Gromov-Hausdorff approximations witnessing the convergence $\frac{1}{r} B_r(p) \ra B_1(o) \subset T_p X$, and these homeomorphisms can be required to preserve the distance from the origin. This result has been refined 
in \cite{K} to show that the homeomorphism can also be assumed to preserve an extremal set. In fact, as we note in Theorem \ref{t:relstability}, we can assume that the homeomorphism will preserve all extremal sets. 
 
 However, requiring that a homeomorphism be a Gromov-Hausdorff approximation controls only its coarse structure. There is very little control on the homeomorphism near the origin. We will construct, for small $r > 0$, a homeomorphism $\frac{1}{r} B_r(p) \ra B_1(o) \subset T_p X$ which has the property that  
when restricted to smaller balls it becomes a better Gromov-Hausdorff approximation. This fine control improves our understanding of the connection between the local and infinitesimal structure.
 
 We begin with a simple generalization of Kapovitch's Relative Stability Theorem \cite{K}.

 \begin{theorem}\label{t:relstability}
 Let $X^n_i$ be a sequence of Alexandrov spaces with curvature uniformly bounded from below, converging to an Alexandrov space $X^n$ of the same dimension. Let $\theta_i: X \ra X_i$ be a sequence of $o(i)$-Hausdorff approximations. Let $A$ be a finite set, and let $\mathcal{E}_i = \{ E^{\alpha}_i \subset X_i \}_{\alpha \in A}$ be a family of extremal sets in $X_i$ indexed by $A$. If $\mathcal{E}_i$ converges to a family of extremal sets $\mathcal{E}$ in $X$, then for all large $i$ there exist homeomorphisms $\theta'_i : (X, \mathcal{E}) \ra (X_i, \mathcal{E}_i)$, $o(i)$-close to $\theta_i$.
 \end{theorem}
 
 We omit the proof of this result. The reader will find it a straightforward exercise to modify the results of \cite{K}, which corresponds to the case that the index set $A$ is a singleton, by taking advantage of the fact that the relative versions of the necessary results of Siebenmann \cite{Si} are developed for the family of all stratified subsets of a stratified space. Let $\mathcal{E}_p$ be the family of $E \in \mathcal{E}$ such that $p \in E$. Then, if $f: X \ra \rrr^k$ is a regular map incomplementable at $p$, there is an admissible map $h$ which, in a neighborhood of $p$, has a unique maximum on each fiber of $f$, providing the conical structure. The key point in adapting the proof is to note that if $S$ is the set of all such maxima, then $S \subset E$ for any $E \in \mathcal{E}_p$.

Next we establish the improved stability homeomorphism from the tangent cone to the ball.

\begin{theorem}\label{t:impstab}
Let $X$ be an Alexandrov space, and let $p \in X$. Then for some $r_0 > 0$ there is a homeomorphism $\theta: B_1(o) \ra \frac{1}{r_0}B_{r_0}(p)$ from the unit ball in the tangent cone to the rescaled ball of radius $r_0$, and there is a function $\epsilon: [0,r_0]  \ra \rrr_{\geq 0}$ with $\lim_{t \to 0}\frac{\epsilon(t)}{t}=0$ such that:
\begin{enumerate}
\item[(i)]$\theta$ preserves the distance from the center of the ball;
\item[(ii)]the pre-image of an extremal set under $\theta$ is extremal; and
\item[(iii)]for any $r \leq r_0$ the restriction of $\theta^{-1}$ to $\frac{1}{r_0}B_r(p)$ is $\epsilon(r)$-close to the rescaled logarithm map and is an $\epsilon(r)$-Gromov-Hausdorff approximation.
\end{enumerate}
\end{theorem}

\begin{proof}
Recall that the logarithm map is a Gromov-Hausdorff approximation from a rescaled metric ball centered at $p \in X$ to the unit ball in $T_p X$ \cite{BBI}. By the Stability Theorem \ref{t:relstability}, as $\frac{1}{r} (X,p) \ra T_p X$, for each small $r$ less than some $r_0$ we can construct homeomorphisms $h_r^{-1}: \frac{1}{r} B_r(p) \ra B_1(o) \subset T_pX$ which approximate the rescaled logarithm map and respect the distance from the origin as well as extremal subsets. For some function $\alpha(t)$ with $\lim_{t \to 0} \alpha(t)=0$, $h_r^{-1}$ is $\alpha(r)$-close to the rescaled logarithm. We will construct $\theta$ by gluing together the various $h_r$.

Define an open cover of $T_pX$ by two sets: $U = B_{7/8}(o)$ and $V = \bar{B}_{5/8}(o)^c$. The gluing theorems used in the proof of the Stability Theorem provide that if we have a $\delta$-Gromov-Hausdorff approximation $B_1(o) \ra Y$ to some space $Y$ which is $\delta$-close to homeomorphisms defined on $U$ and on $V$, we may glue these to give a homeomorphism which is $\kappa(\delta)$ close to the Gromov-Hausdorff approximation for some function $\kappa$ with $\lim_{t \to 0}\kappa(t)=0$. Furthermore, if the initial homeomorphisms respect the distance from $o$, the glued homeomorphism can be chosen to do the same.

Pick a sequence $r_n = r_0 2^{-n}$. Let $n(t)$ be the greatest integer which is $\leq -\log_2 (t/r_0)$ and then define $\epsilon(t) = 2^{-n(t)}\kappa(2\alpha(r_{n(t)-1}))$. 
We will now construct $\theta$ inductively.

 We define the map $\theta_0 = h_{r_0}$. It is an $\alpha(r_0)$-approximation. Assume that the map $\theta_{n-1} : B_1(o) \ra \frac{1}{r_0} B_{r_0}(p)$ has been defined and is an $\epsilon(t)$-approximation to the logarithm on balls of radius $t \geq 2^{-n}$ and a $2^{-(n-1)}\alpha(r_{n-1})$-approximation on the ball of radius $2^{-n}$. Assume further that it respects the distance from the center. Consider the rescaled map $\theta_{n-1} : B_{2^n}(o) \ra \frac{1}{r_n}B_{r_0}(p)$, and restrict this map to $V$. Then we can glue this map to the restriction of $h_{r_{n}}$ to $U$ while respecting the distance from the center, and the inverse of the resulting map will be $\kappa(2\alpha(r_{n-1}))$-close to the rescaled logarithm map, while being $\alpha(r_n)$-close on the ball of radius $1/2$.

Shrinking back down, we have defined a new homeomorphism $\theta_{n} : B_1(o) \ra \frac{1}{r_0} B_{r_0}(p)$. Outside of $B_{7r_n/8}(o)$ we have $\theta_n = \theta_{n-1}$ while inside $B_{5r_n/8}(o)$ we have $\theta_n = h_{r_{n}}$, after rescaling. Therefore, as an approximation of the logarithm, the accuracy of $\theta^{-1}_n$ is $\epsilon(t)$ on balls of radius $t \geq r_n/2$, and on the ball of radius $r_n/2$ it is a $2^{-n}\alpha(r_n)$-approximation.

For any point $q \neq o$, there is a small neighborhood of $q$ on which $\theta_n = \theta_{n+1}$ for sufficiently large $n$, so the limit map is a well-defined homeomorphism satisfying the conditions.
\end{proof}
 
 \subsection{Blow-up Construction}

\begin{definition}Let $X$ be an Alexandrov space, and let $p \in X$. The \emph{blow-up} of $X$ at $p$ is defined to be the space $(\Sigma_p X \times [0,1]) \cup_{\phi} X \setminus \{ p \}$ where $\phi : \Sigma_p X \times (0,1] \ra X$ is given by $\phi(v,t) = \gexp_p(tv)$. We will denote it by $\blowX$.\end{definition}

The blow-up $\blowX$ is the result of removing the point $p$ and replacing it with the space of directions, $\Sigma_p$. Each direction is added as the limit point of the corresponding radial curve starting at $p$ in that direction. We will refer to each such point as a \emph{blow-up point}. 

Note that the blow-up is a purely topological construction. It carries no natural metric. However, it does have a natural partition arising from the stratified structure of $X$. Given an extremal set $E$ containing $p$, the closure of $E$ in $\blowX$ is $E \cup \Sigma_p E$, and we let this be a part of $\blowX$. Observe that it is possible that additional extremal sets might arise in $\Sigma_p$ which are not related to extremal sets in $X$, but these are not relevant for our considerations. 

\begin{proposition}
Let $X$ be an Alexandrov space, $p \in X$, and $r_0 > 0$ be small. Then $\Sigma_pX \times [0,1]$ is homeomorphic to the blow-up of the closed ball $B_{r_0}(p)$. The natural partition of the blow-up is given by the product of a stratification of $\Sigma_pX$, which may be coarser than the 
stratification by extremal sets, with the interval $[0,1]$.
\end{proposition}

\begin{proof}
Choose $r_0>0$ so small that the conclusions of Theorem \ref{t:impstab} hold, and we have a homeomorphism $\theta:B_0(1) \ra B_{r_0}(p)$. By removing the origin of the tangent cone, we have a homeomorphism from $\Sigma_p X \times (0,1]$ to the complement of the blow-up points. Any extremal sets in $\Sigma_p X$ which arise from extremal sets in $X$ will be mapped onto the blow-up partition. Combine this with the obvious map from $\Sigma_p X \times \{0\}$ to the blow-up points to obtain a function $\tilde{\theta}:\Sigma_p X \times [0,1] \ra \blowX$. We need only confirm the continuity of this function at the blow-up points.

Let $(w,t) \ra (v,0)$ in $\Sigma_p X \times [0,1]$. Because $\phi^{-1}$ approximates the rescaled logarithm, we have that if $\theta(w,t) = y$ then $\frac{1}{t} \log_p(y)=u$ is $\epsilon(t)$-close to $w$, where $\lim_{r \to 0} \epsilon(r) = 0$. In other words, $\theta(w,t) = \gexp_p (tu)$ for some $u$ with $\dist(u,w) \leq \epsilon(t)$.

As 
$(w, t) \ra (v, 0)$, 
we have that $u \ra v$, so $\phi(w,t)$ converges to the blow-up point $v$. 
\end{proof}

We can also blow up a $G$-space $X$ around an entire orbit. We define $\GblowX$, the blow-up of $X$ at $G(p)$, to be 
\bdm
G \times_{G_p} \nu_p \times [0,1] \cup_{\phi} X \setminus G(p)
\edm
where the map $\phi : G \times_{G_p} \nu_p \times (0,1] \ra X \setminus G(p)$ is given by the normal gradient exponential $\phi(g, v, t) = \gexp_{G(p)}(t \cdot g_{*}(v))$. The blow-up admits an action of the entire group $G$.

\begin{proposition}\label{p:orbitblowup}
Let $G$, a compact Lie group, act isometrically on an Alexandrov space $X$, and let $p \in X$. Then $\GblowX / G \cong \bar{X}^{!}_{\bar{p}}$, and the stratification of $\GblowX / G$ by orbit type is finer than the blow-up stratification.
\end{proposition}

\begin{proof}
$\GblowX$ is defined by gluing together two spaces using a $G$-equivariant map $\phi$. The orbit space $\GblowX/G$ is therefore also given as a union $\nu / G_p \times [0,1] \cup_{\bar{\phi}} \bar{X} \setminus \{\bar{p}\}$, where $\bar{\phi}$ is the map induced by $\phi$. Because $\dist(G(p), \cdot)$ descends to $\dist(\bar{p},\cdot)$ the map $\bar{\phi}$ is $\gexp_{\bar{p}}$ so that $\GblowX/G$ is the blow-up $\bar{X}^{!}_{\bar{p}}$.

The stratification by orbit type of $\bar{X}$ is coarser than that by extremal sets \cite{PP1}, and by considering the action of $G_p$ on quasigeodesics emanating from $p$ we can see that the orbit-type stratification of $\Sigma_{\bar{p}}$ is the infinitesimal version of that of $\bar{X}$.
\end{proof}

\subsection{Proof of the Slice Theorem}

For the purposes of this proof, we recall the following important result of Palais \cite{Pa} (cf. Theorem 7.1 of \cite{Br}).

\begin{cht}\label{t:cht}
Let $Y$ be a topological space such that every open subspace is paracompact, let $I$ be a closed interval, and let $G$ be a compact Lie group. Suppose that $W$ is a $G$-space with orbit space $W/G \cong Y \times I$ such that the orbit structure is given by a product, that is, for each $y \in Y$ the orbit types are constant on $\{y\} \times I$. Then there is a $G$-space $X$ with $X/G \cong Y$ so that $W$ is equivalent to $X \times I$, where $G$ acts trivially on $I$.
\end{cht}

We can now proceed to the proof of the main theorem of this section.

\begin{slicethm} \label{l:sliceinf} Let a compact Lie group $G$ act isometrically on an Alexandrov space $X$. Then for all $p \in X$, there is some $r_0 > 0$ such that for all $r < r_0$ there is an equivariant homeomorphism $$\Phi : G \times_{G_p} K\nu_p \rightarrow B_r(G(p)),$$ where $\nu_p$ is the space of normal directions to the orbit $G(p)$.\end{slicethm}

\begin{proof}
Let $r_0$ be such that $B_{r_0}(\bar{p}) \subset X/G$ is homeomorphic to $T_{\bar{p}} (X/G)$. Let $W$ be the portion of $\GblowX$ corresponding to $B_{r_0}(G(p))$. Then by Proposition \ref{p:orbitblowup}, the orbit space $\bar{W}$ is the blow-up of $B_{r_0}(\bar{p}) \subset \bar{X}$ at $\bar{p}$.

Now let $S$ be a slice at $p$, the existence of which is guaranteed by \cite{MY57}. We may assume that $B_{r_0}(\bar{p}) = S/G_p$. Let $S^* = S \setminus \{p\}$. Let $V = S^* \cup \nu_p \subset W$. $V$ is $G_p$-invariant. We claim that $V/G_p \cong \bar{W}$, and so has a product structure. Assuming the claim, the Covering Homotopy Theorem 
shows that $S^*$ is $G_p$-equivariant to $\nu_p \times (0,r_0)$, and adding the point $p$ back to $S^*$ gives the result.

To prove the claim, note first that $V$ is a closed subspace of $W$. The only sequences in $S^*$ which do not converge are those which are leaving $B_{r_0}(G(p))$ or which, in $S$, have limit $p$, so the closure of $S^*$ in $W$ is made up of the union of $S^*$ with some horizontal directions. If $q \neq p$ and $\xi \in \nu_q$ is in $\cl(S^*)$, then any open neighborhood of $\xi \in W$ will contain elements of $S^*$. Any open neighborhood $U$ of $q$ in $X$ contains an open neighborhood of $\xi \in W$, but we can clearly make $U$ so small that it does not intersect $S$. Therefore $\cl(S^*) \subset \nu_p \cup S^*$. We can write $V$ as the union of two closed sets $\cl(S^*) \cup \nu_p$, so $V$ is closed.

The projection $W \ra \bar{W}$ is a map from a compact space to a Hausdorff space, and so a closed map. Its restriction to the closed subspace $V$ is also closed, and it is surjective, so it is a quotient map. But clearly it takes $G_p$-orbits on $V$ to points, so it is in fact the orbit-space projection, and therefore $V/G_p \cong \bar{W}$ as claimed.
\end{proof}

The slice can be explicitly constructed as follows.
Let $\pi: X\ra \bar{X}=X/G$. Let $p \in X$, and let $\bar{p} = \pi(p) \in \bar{X}$. Let $\bar{h}: \bar{X} \ra \rrr$ be a function built up from distance functions such that it is strictly concave on some $U$ containing $\bar{p}$ and attains its maximum value on $U$ at $\bar{p}$ (see \cite{PP1} for the construction, cf.\ \cite{GW1, K2} for more explicit constructions). Let $r_0$ be such that $B_{r_0}(\bar{p}) \subset \bar{h}^{-1}\left( \left[ a, \infty \right) \right)  \subset U$ for some $a$. The gradient flow of $\bar{h}$ gives the retraction $\bar{F}: B_{r_0}(\bar{p}) \ra \left\lbrace p \right\rbrace$.

We can lift the function $\bar{h}$ to a function $h: X \ra \rrr$. This function is defined by distance functions from $G$-orbits, and so its gradient flow gives a $G$-equivariant retraction $F : B_r(G(p)) \rightarrow G(p)$. Then by Proposition II.3.2 of \cite{Br}, $F^{-1}(p)$ is a slice.

Finally, in the special case where $X$ has strictly positive curvature and $\bar{X}$ has non-empty boundary, the Soul Theorem \ref{t:soulthm} tells us that the soul of $\bar{X}$ is a point $\bar{p}$. 
If $G(p)$ is the corresponding orbit in $X$,  we may use the Sharafutdinov retraction in place of $\bar{F}$ to construct the slice, and show that $\pi^{-1} (\bar{X} \setminus \d \bar{X})$ is equivariantly homeomorphic to $G \times_{G_p} K\nu_p$.

\section{Fixed Point Sets}\label{s:fixedpoints}
 
In this section, we show that some important results from the theory of isometric actions of compact Lie groups on Riemannian manifolds still hold in the context of Alexandrov geometry.

We show that, in analogy to the Riemannian situation, the fixed point set of a group action is a totally quasigeodesic subset. We then consider torus actions on positively curved spaces, and show that their fixed point sets are always of even codimension, that in even dimensions the fixed point set is always non-empty, and that in odd dimensions if no point is fixed, then there must be a circle orbit.

\subsection{Structure of Fixed Point Sets}

\begin{proposition} Let $G$ be a compact Lie group acting on an Alexandrov space $X$ by isometries. Let $H \subset G$ be a closed subgroup, and let $F \subset X$ be the set of fixed points of $H$. Then $F$ is a totally quasigeodesic subset of $X$ and admits a stratification into manifolds.\end{proposition}

\begin{proof}
The isometric image of $F$ in the orbit space $X/H$ is an extremal set, and therefore it is totally quasigeodesic and stratified into manifolds \cite{PP1}. For any $p \in X$, the function $\dist(p,\cdot)$ on $F$ is equal to $\dist(\bar{p},\cdot)$ on the image of $F$, and so curves in $F$ which are quasigeodesics for $X/H$ are also quasigeodesics for $X$, giving the result. 
\end{proof}

\begin{example}\label{e:CP2}
Suspend an isometric $T^1$-action on $\ccc P^2$, where $T^1$ fixes an $S^2$ and an isolated point in the $\ccc P^2$. $\Fix(\Sigma \ccc P^2; T^1)$ is connected and consists of an $S^3$ and an interval, $I$, where the interval's endpoints are the antipodes of $S^3$. The strata are (i) a twice punctured $S^3$, (ii) an open interval, and (iii) two isolated points.
\end{example}

\subsection{Torus Actions on Positively Curved Spaces}

We first recall Petrunin's generalization of Synge's Lemma, which is used to prove the Generalized Synge's Theorem \ref{t:GST}.

\begin{GSlemma}\label{l:gensynge}\cite{Pet1} Let $X$ be an orientable Alexandrov space with  
$\curv \geq 1$ and let $T:X\rightarrow X$ be an isometry. Suppose that
\begin{enumerate}
\item $X$ is even-dimensional and $T$ preserves orientation; or
\item $X$ is odd-dimensional and $T$ reverses orientation.
\end{enumerate}
Then $T$ has a fixed point.
\end{GSlemma}

In the even-dimensional case, the theory of ramified orientable double covers yields the following corollary.

\begin{corollary}\label{c:synge}Let $X$ be an Alexandrov space of even dimension with $\curv \geq 1$, and let $G$ be a connected Lie group acting on $X$ by isometries. Then for any $g \in G$, $g$ has a fixed point.\end{corollary}

\begin{proof}
We may take $X$ to be non-orientable. By Lemma \ref{l:lift} we may lift the action of $G$ to an action of a 2-fold covering group $\tilde{G}$ on $\tilde{X}_{\txt{Ram}}$, and then the Generalized Synge's Lemma \ref{l:gensynge} applies to a lift of $g$ which is in the connected component of the identity of $\tilde{G}$.
\end{proof}

We can now prove two results about the isotropy of torus actions: the counterparts of Berger's Lemma \cite{Berger} and its odd-dimensional analogue \cite{Sugahara}.

\begin{lemma}\label{l:fixedpoint}Let $T^k$ act by isometries on $X^{2n}$, an even-dimensional space of positive curvature. Then $T^k$ has a fixed point.\end{lemma}

\begin{proof}Consider a dense 1-parameter subgroup of $T^k$, and within it an infinite cyclic subgroup. By Corollary \ref{c:synge}, the cyclic subgroup fixes a point. As we move the generator of the subgroup towards the identity, we generate a sequence of fixed points in $X$, and any limit point of that sequence will be fixed by the torus.\end{proof}

\begin{corollary}\label{c:odddim}Let $T^k$ act by isometries on $X^{2n+1}$, an odd-dimensional space of positive curvature. Then either there is a circle orbit or $T^k$ has a fixed point set of dimension at least one.\end{corollary}

\begin{proof}If $T^k$ has a fixed point $p$, then we may apply Lemma \ref{l:fixedpoint} to the isotropy action
on $\Sigma_p$, a positively curved Alexandrov space of even dimension. Otherwise, let $T^1 \subset T^k$ act non-trivially, and consider the induced action of $T^{k-1}$ on the $2n$-dimensional space $X / T^1$. By Lemma \ref{l:fixedpoint}, this action fixes a point, and that point corresponds to a circle orbit of $T^k$ in $X$.\end{proof}

Finally, we note that an easy induction shows that a familiar result on the codimension of the fixed point set of circle actions, or, more generally, torus actions, on Riemannian manifolds holds for Alexandrov spaces.

\begin{proposition}\label{p:evencodim} Let $T^1$ act isometrically on $X^n$, a compact Alexandrov space. Then the fixed point set components of the circle action are of even codimension in $X^n$. \end{proposition}

\section{Fixed-Point Homogeneous Actions}\label{s:fph}

One measurement for the size of a transformation group $G\times X\rightarrow X$ is the dimension of its orbit space $X/G$, also called the {\it cohomogeneity} of the action. The cohomogeneity is clearly constrained by the dimension of the fixed point set $X^G$  of $G$ in $X$. 
In fact, $\dim (X/G)\geq \dim(X^G) +1$ for any non-trivial action. In light of this we have the following definition.

\begin{definition}The {\it fixed-point cohomogeneity} of an action, denoted by $\Cohomfix(X;G)$, is given by
\[
\textrm{cohomfix}(X; G) = \dim(X/G) - \dim(X^G) -1\geq 0.
\]
\end{definition}

A space with fixed-point cohomogeneity $0$ is called {\it fixed-point homogeneous}.
Note that in the manifold case,  
the largest connected component of the fixed point set corresponds to a connected component of the boundary of the orbit space. However, 
for such an action on an Alexandrov space, it may only be a subset of a connected component of the boundary, as is the case, for example, of the suspended circle action on $\Sigma \rrr P^2$.

We first recall the classification result for fixed-point homogeneous Riemannian manifolds of positive sectional curvature \cite{GS2}.

\begin{theorem} Let $G$, a compact Lie group, act isometrically and fixed-point homogeneously on $M^n$, a closed, simply-connected, positively curved Riemannian
manifold with $M^G \neq \emptyset$. Then $M^n$ is diffeomorphic to one of $S^n$, $\ccc P^k$, $\hh P^m$ or ${\rm Ca}P^2$, where $2k=4m=n$.
\end{theorem}

This result is obtained via the following structure theorem from \cite{GS2} which allows us to decompose the manifold as the union of two disk bundles.

\begin{theorem}[Structure Theorem]\label{t:fphstructure} Let $M$ be a positively curved Riemannian manifold with an (almost) effective isometric
fixed-point homogeneous $G$-action and $M^G\neq \emptyset$. If $F$ is the component of $M^G$ with maximal dimension then the following hold:
\begin{itemize}
\item[(i)] There is a unique orbit $G(p)\cong G/G_p$ at maximal distance to $F$ (the ``soul" orbit).
\item[(ii)] All $G_p$-orbits in the normal sphere $S^l$ to $G(p)$ at $p$ are principal and diffeomorphic to $G_p/H$. Moreover $F$ is diffeomorphic to $S^l/G_p$.
\item[(iii)] There is a $G$-equivariant decomposition of $M$, as 
\bdm
M=D(F)\cup_E D(G(p)),
\edm
where $D(F)$, $D(G(p))$, are the normal disk bundles to $F$, $G(p)$, respectively, in $M$ with common boundary $E$ when viewed as tubular neighborhoods.
\item[(iv)] All orbits in $M\setminus (F\cup G(p))$ are principal and diffeomorphic to $S^k\cong G/H$, the normal sphere to $F$.
\end{itemize}
\end{theorem}

This result is somewhat similar to the situation of positively curved 
Riemannian manifolds of cohomogeneity one, which also admit a decomposition as a union of disk bundles. A similar structure is observed when passing to 
positively curved Alexandrov spaces of cohomogeneity one, that is, they admit a decomposition as a union of more general cone bundles \cite{GGS}.

One might expect, therefore, that positively curved fixed-point homogeneous Alexandrov spaces would also admit a decomposition as  a union of cone bundles. However, when the normal spheres of Theorem \ref{t:fphstructure} are replaced by more general spaces of positive curvature, the additional possibilities
mean that the
structure of the orbit space can be more complicated than in the Riemannian setting. In particular, part (ii) of the result fails and non-principal orbits may appear in the complement of $F \cup G(p)$.

\begin{example}
Consider $\Sigma \rrr P^2$ with the suspended $T^1$ action. 
$\Fix(\Sigma \rrr P^2; T^1)=I$, 
and the soul orbit is $T^1(p)\cong S^1/\zzz_2$. 
 At a point in the soul orbit, the $\zzz_2$ isotropy action on the circle of normal directions fixes two points, giving rise to an $S^2$ 
with $\zzz_2$ isotropy. The $S^2$ intersects 
 the endpoints of $I$ 
in its antipodes.
Although $\Sigma \rrr P^2$ can be written as the union of two cones, it is not a union of cone bundles over 
$\Fix(\Sigma \rrr P^2; T^1)$ and the soul orbit $T^1(p)$.
\end{example}

Theorem \ref{t:expandedC} below, an expanded version of Theorem C from the Introduction, provides a description of positively curved fixed-point homogeneous Alexandrov spaces as 
a join of a space of directions and a compact, connected Lie group, $G$,
modulo a subgroup $K\subset G$. This provides an alternative to the disk-bundle method of viewing fixed-point homogeneous Riemannian manifolds of positive curvature. 

\begin{theorem}\label{t:expandedC} 
Let a compact Lie group $G$ act isometrically and fixed-point homogeneously on $X^n$, a compact $n$-dimensional Alexandrov space of positive curvature and assume that $X^G\neq \emptyset$.
If $H \subset G$ is the principal isotropy and $F$ is the component of $X^G$ with maximal dimension then the following hold:
\begin{itemize}
\item[(i)] There is a unique orbit $G(p)\cong G/G_p$ at maximal distance from $F$ (the ``soul" orbit).
\item[(ii)] All principal $G_p$-orbits in $\nu$, the normal space of directions to $G(p)$ at $p$,   are homeomorphic to $G_p/H$.  Moreover $F$ is homeomorphic to $\nu / G_p$.
\item[(iii)] The space $X$ is $G$-equivariantly homeomorphic to  
\bdm
(\nu *G)/G_p,
\edm
where $G_p$ acts on the left on $\nu*G$, the action on $\nu$ being the isotropy action at $p$ and the action on $G$ being the inverse action on the right.
The $G$-action is induced by the left action on $\nu * G$ given by the join of the trivial action and the left action.
\item[(iv)] The principal orbits in $X\setminus (F \ \cup \ G(p))$ are homeomorphic to $\nu(F)\cong G/H$, where $\nu(F)$ is the positively curved space of normal directions to $F$.
\end{itemize}
\end{theorem}

For every positively curved space $\nu$, its join to a positively curved homogeneous $G$-space yields a fixed-point homogeneous $G$-space. It follows that the set of all simply-connected positively curved  spaces of dimension $n$ can be placed in bijection with the fixed-point homogeneous simply-connected positively curved spaces of dimension $m$ for any $m > n$. This is in stark contrast to the Riemannian result, where only the compact rank one symmetric spaces arise.

For the proof of Theorem C, we introduce a notion of regular points in extremal sets. We say that a point $p\in E^k \subset X^n$ is \emph{$E$-regular} when $\Sigma_p E=S^k_1$, that is, the space of directions tangent to $E$ at the point $p$ is isometric to a unit round sphere.

\begin{lemma} Let $X$ be an Alexandrov space with boundary and $E=\d X$. Then $E$ contains a dense set of $E$-regular points.\end{lemma}

\begin{proof}Consider $D(X)=X\cup_{\d X} X$, the double of $X$, formed by identifying two copies of $X$ along their common boundaries, $\d X$. Since $D(X)$ is also an Alexandrov space 
\cite{P}, it has a dense set of regular points, and, since that set is convex, $\d X \subset D(X)$ also has a dense set of regular points. Then in $X$ these points will be $E$-regular.
\end{proof}

\begin{proof}[Proof of Theorem C] The maximal connected component of the fixed point set, $F$, has codimension 1 in $X/G$ and so corresponds to a union of faces in the boundary. Part (i) now follows from the Soul Theorem  \ref{t:soulthm}  applied to the quotient space $X/G$, a positively curved Alexandrov space, retracting from the faces which make up $F$. Let $\nu$ be the space of normal directions to the orbit $G(p)$. Part (ii) follows from the Slice Theorem \ref{l:sliceinf}, noting also that $F$ is homeomorphic to the space of directions at the soul point of $X/G$, which is $\nu / G_p$.

For part (iii), $X \setminus F$ is homeomorphic to $(K(\nu) \times G) / G_p$ by the Slice Theorem, and this homeomorphism is $G$-equivariant, where $G$ acts trivially on $\nu$ and 
on $G$ by its left action. We may write $K(\nu) \times G$ as $\nu \times G \times (0,1]$, where $\nu \times G \times \{ 1 \}$ is identified to $G$. The set $F$ is, by part (ii), homeomorphic to $\nu / G_p$, and it is fixed by $G$, so the entire space $X$ is in fact homeomorphic to $( \nu \times G \times [0,1] ) / G_p$, where $\nu \times G \times \{ 0 \}$ has been identified to $\nu$ and $\nu \times G \times \{ 1 \}$ to $G$, and the homeomorphism is $G$-equivariant. 

For part (iv), let $\bar{p} \in F \subset X/G$ be $F$-regular, and hence the image in $\nu / G_p$ of a principal $G_p$-orbit. Then at $p \in F \subset X$ we have an isometry $\Sigma_p X = S^{k-1} * \nu(F)$ where $k = \dim(F)$ and $\nu(F)$ is the normal space to $F$. There is a neighborhood of $p$ which is comprised entirely of principal orbits and of fixed points in $F$. Since by assumption $G$ acts transitively on $\nu(F)$, it follows that  $\nu(F)$ is given by $G/H$.
Hence, $G/H$ has curvature bounded below by 1, completing the proof of the theorem.
\end{proof}

In the special case where $H\normal G_p\subset G$, we have the following corollary.

\begin{corollary}\label{c:fph} Let $G$ act isometrically and fixed-point homogeneously on $X^n$, an $n$-dimensional, closed Alexandrov space of positive curvature and assume that $X^G\neq \emptyset$. Suppose further that $H\normal G_p\subset G$, where $H$ is the principal isotropy of the $G$-action and $G_p$ is the isotropy subgroup of the ``soul" orbit. Then $X^n$ is equivariantly homeomorphic to 
\bdm
(\nu * G/H)/K,
\edm
where $\nu$ is the normal space of directions to $G/G_p$ and $K \cong G_p / H$. That is, $X^n$ is homeomorphic to the quotient of the join of two positively curved Alexandrov spaces.
\end{corollary}

\begin{proof} Write $K$ for $G_p / H$. By Theorem C, $X^n$ is homeomorphic to 
\bdm
(\nu * G)/G_p \cong ((\nu * G)/H)/K.
\edm
Since $H$ is the principal isotropy of the $G_p$-action on $\nu$, and is a normal subgroup of $G_p$, it acts trivially on $\nu$. Therefore we have 
\bdm
(\nu * G)/H \cong \nu * G/H,
\edm
and the result follows.
\end{proof}

Note that for any space $\nu$ of $\curv \geq 1$, we can join $\nu$ to any homogeneous $G$-space of $\curv \geq 1$ to obtain a positively curved Alexandrov space with a fixed-point homogeneous $G$-action. In this sense, we can think of fixed-point homogeneous spaces as being plentiful among positively curved Alexandrov spaces.

Observe, however, that if we restrict our attention to positively curved Riemannian $n$-dimensional manifolds and assume that $G_p\neq G$, then,  with the unique exception of the fixed-point homogeneous $Spin(9)$-action on $\textrm{Ca}P^2$, $H\normal G_p \subset G$. Hence, Corollary \ref{c:fph} allows us to  represent all such manifolds as
\bdm
M^n \cong (S^m * G/H)/ K \cong (S^m* S^l)/K \cong S^{m+l+1}/K,
\edm
where $K \cong G_p/H$, and $K$ is one of 
either $SU(2)$, $N_{SU(2)}(T^1)$, $T^1$, or a
finite subgroup of $O(n+1)$ (cf. \cite{Br}). That is, all these manifolds are spheres 
or the base of a fibration whose total space is a sphere. $\textrm{Ca}P^2$ cannot be written as the base of such a fibration, and its decomposition
as a join is given by
\bdm
(S^7*Spin(9))/Spin(8).
\edm

We also note that in the special case where $G_p=G$, the principal
isotropy subgroup is almost always \emph{not} normal in $G$.
This is not an issue for the Riemannian case, though, because the only
groups that can act principally on $n$-spheres must either act
transitively, in which case, the decomposition as a join gives us that
\bdm
(\nu * G)/G_p \cong S^0 * G/H \cong S^0 *S^l = S^{l+1},
\edm
or they must act
freely, as in the examples discussed above.

Finally, we note that in analogy to the Riemannian case, we can decompose
a fixed-point homogeneous, positively curved Alexandrov space
as a union of cone bundles when we assume that all orbits in the complement of
the fixed point set $F$ and the soul orbit $G(p)$ are principal. 

\section{Maximal Symmetry Rank}\label{s:msr}

In this section we show that, as in the Riemannian case, the bound on the symmetry rank of a positively curved Alexandrov space of dimension $n$ is $\lfloor \frac{n+1}{2}\rfloor$, and that when the bound is achieved some circle subgroup acts fixed-point homogeneously. We then inductively use the join description of fixed-point homogeneous spaces to show that all such spaces are quotients of spheres, as stated in Theorem D.

\begin{theorem}\label{t:upperbound}  Let $T^k$ act isometrically and (almost) effectively on $X^n$, a positively curved Alexandrov space. Then, 
\bdm 
k\leq \lfloor \frac{n+1}{2}\rfloor.
\edm
Further,   in the case of equality, for some $T^1\subset T^k$, $\codim(\Fix(X^n; T^1))=2$.
\end{theorem}

\begin{proof}
If $X$ has boundary then the action is determined by the isotropy at the soul, and if $X$ is non-orientable the action will lift to the ramified orientable double cover. Therefore, we may assume that $X$ is closed and orientable.

The proof is by induction on the dimension $n$ of the space. When $n = 1$, the maximal torus action is given by
$T^1$ acting on $S^1$, fixing the empty set, $\emptyset$, which has codimension 2. The crux of the induction step is that where a group acts effectively, the action of any isotropy group on the normal space to its orbit must also be effective \cite{GGG}. 

If $n=2k-1$, then by the inductive hypothesis, an effective action of $T^k$
cannot fix points, and so Corollary \ref{c:odddim} implies that the action has a circle orbit. If $n=2k$, then Lemma \ref{l:fixedpoint} implies that the action has a fixed point.

Aiming for a contradiction, we suppose that $T^{k+1}$ acts on $X^n$, with $n=2k-1$ or $2k$. Consider the isotropy action at a circle orbit or at a fixed point, respectively. By the inductive hypothesis, this action cannot be effective. This proves the bound on the rank. If $T^k$ acts, we again consider the isotropy action at a circle orbit or fixed point. This action is also of maximal rank and so, by the inductive hypothesis, there is a circle subgroup of the isotropy which fixes a set of codimension 2 in the normal space, and hence,
a set of codimension 2 in $X$.
\end{proof}

Before continuing with the proof of Theorem D, we present a few examples of how to construct positively curved Alexandrov spaces with maximal symmetry rank.

\begin{example} $S^n(1)$ has isometry group $O(n+1)$, and this group is clearly of maximal rank. 
The action of the maximal torus in $O(n+1)$ may be considered as 
the prototypical maximal rank action in positive curvature.\end{example}

\begin{example} Let $\Gamma$ be a finite subgroup of the maximal torus in $O(n+1)$. 
Clearly $S^n / \Gamma$ admits an action of maximal symmetry rank. If 
$\Gamma$  
acts freely, we obtain the lens spaces (including the odd-dimensional real projective spaces) \cite{Smith}. Note that in this category of examples, the
$2n$-dimensional spaces are simply suspensions of the
$(2n-1)$-dimensional spaces.\end{example}

It is worth noting at this point that the natural sufficient condition for $T^k$, a torus of maximal rank, to act on $S^n / H$ is that $T^k \subset N(H)$, the normalizer of the group. The following lemma shows that when $\dim(H) \leq 2$, so that $H$ is virtually abelian,
$T^k \subset N(H) \iff T^k \subset Z(H)$, and so the condition is equivalent to $H \subset Z(T^k)$.

\begin{lemma}\label{l:norm=cent}Let $G$ be a compact Lie group, and let $H \subset G$ be a closed virtually abelian subgroup. Then the connected components of the identity of the normalizer and centralizer of $H$ agree, that is, $N_0(H) = Z_0(H)$.\end{lemma}

\begin{proof} There is a smooth map $N(H) \ra Aut(H)$ sending any element to conjugation by that element. Since $H_0$ is, by definition, abelian, it is a torus or a point, and so $Aut(H)$ is discrete. Therefore $N_0(H)$ is mapped to the identity automorphism, and we have $N_0(H) \subset Z_0(H)$. The reverse inclusion is trivial.\end{proof}

\begin{example} Let $\Gamma\subset Z(T^n)\subset O(2n+1)$. Since 
$Z(T^n)\neq T^n$ in $O(2n+1)$, we may choose $\Gamma$ to be a finite subgroup not contained in $T^n$.
If we require $\Gamma$ to act freely, it must be the antipodal map and we obtain $\rrr P^{2n}$. 
However, if we allow $\Gamma$ to have non-trivial isotropy, 
we will obtain spaces which are locally non-orientable, such as $S^4 / \zzz_2 \cong \rrr P^2 * S^1$, where the involution fixes a circle, as well as 
spaces with boundary, such as $D^n=S^n/\zzz_2$, where $\zzz_2$ is simply a reflection.
\end{example}

\begin{example}Let $\Gamma \subset T^{n+1} \subset O(2n+2)$ be a rank one subgroup. Then the $2n$-dimensional space $S^{2n+1} / \Gamma$ admits a $T^n$-action. In particular, if $\Gamma$ acts freely, and is therefore the diagonal circle, then we obtain a $T^n$-action on $\ccc P^n$.\end{example}

We see that many spaces of maximal symmetry rank can be obtained in the same way as the Riemannian examples, that is, by taking quotients of spheres. We obtain non-Riemannian spaces  
by allowing the group action to have isotropy. Theorem D shows that all such spaces arise in this way.

\begin{thmd}\label{t:newmsr} Let $X$ be an $n$-dimensional, compact,
 Alexandrov space with $\curv \geq 1$ admitting an isometric, (almost) effective $T^k$-action. Then $k\leq \lfloor \frac{n+1}{2} \rfloor$, and in the case of equality, either
\begin{enumerate}
\item $X$ is a spherical orbifold, homeomorphic to $S^n / H$, where $H$ is a finite subgroup of the centralizer of the maximal torus in $O(n+1)$, or
\item only in the case that $n$ is even, $X \cong S^{n+1}/H$, where $H$ is a rank one subgroup of the maximal torus in $O(n+2)$.
\end{enumerate}
In both cases the action on $X$ is equivalent to the induced action of the maximal torus on the $H$-quotient of the corresponding sphere. 
\end{thmd}

\begin{proof}
The bound on the rank has already been shown. Let us assume that $X^n$ is closed and orientable, and that $T^k$ acts with maximal rank on $X$. Since 
Alexandrov spaces of dimension less than or equal to two are topological manifolds, we will assume that $n\geq 3$. By Theorem \ref{t:upperbound}, there is some $T^1 \subset T^k$ which acts with a fixed point set of codimension 2. Let $F$ be a codimension two, connected component of that fixed point set. Then there is a unique orbit, $T^1 (p)$, at maximal distance from $F$. Let $\nu_p$ be the normal space at the point $p$ to this orbit. In $X / T^1$, this orbit becomes a point fixed by $T^k / T^1$. Either $T^1_p = T^1$, in which case $p$ is fixed by the entire torus, or $T^1_p$ is finite, in which case $T^k(p)$ is a circle orbit. In either case, the isotropy action of $T^k_p$ on $\nu_p$ is once again an action of maximal symmetry rank upon a closed orientable positively curved Alexandrov space.

Now by Theorem C, we may write $X \cong (\nu * S^1)/T^1_p$. We proceed inductively, until $\nu$ is given by $S^1$ or $S^0$, to see that $X$ is homeomorphic to the quotient of some sphere $S^m$ by a subgroup $H$ of a linearly acting torus, and that the torus action on $X$ is induced by the action of the torus on $S^m$. It is clear from considering the dimension of the space and the rank of the torus that either $m=n$ and $\rk (H) = 0$, or, only in the case where $n$ is even, $m=n+1$ and $\rk (H) = 1$. 
Further, one easily sees 
that the orbit space $X/T^k$, stratified by isotropy type, is either a simplex or a suspended simplex.

In the case that $X^n$ is not orientable, we have $X = (S^m/H)/\zzz_2$, where $\zzz_2$ reverses orientation. It follows that $(S^m/H)/T^k$ is a suspended simplex, and $\dim (X) = m = n$ is even with $H$ finite. The action of $\zzz_2$ can then be lifted to the sphere, where it is the composition of an element of the torus with a reflection. In the case that $X$ has boundary, the isotropy action at the soul determines the action on $X$. $X$ is the cone on an odd-dimensional maximal symmetry rank space, or, equivalently, the quotient by a reflection of the suspension of an odd-dimensional maximal symmetry rank space. 
\end{proof}

It follows that all odd-dimensional maximal symmetry rank spaces are orientable. Note that this result is sharp. For example, $\Sigma(\rrr P^{2n})$ admits an action of almost maximal symmetry rank, that is, of rank $\lfloor \frac{n-1}{2} \rfloor$.
We can see from the inductive nature of the proof that, if the space is odd-dimensional, it is 
built up by an iterated process, joining a circle to a space of lower dimension and then taking a quotient by a finite group.  That is, one can 
write the space as 
\bdm
(\cdots(S^1/\Gamma_1 * S^1)/\Gamma_2 *   \cdots * S^1)/\Gamma_k,
\edm
where $\Gamma_i$ are finite cyclic subgroups of $T^k$, $1\leq i \leq k$. 
If it is even-dimensional and orientable, then it is a suspension of an odd-dimensional example, or it is the quotient of an odd-dimensional example by a circle. If it is not orientable, or has boundary, then it is a quotient by an involution on an example of suspension type which interchanges the poles of the suspension.

This view of the proof allows us to show the following result on the fundamental groups and Euler characteristics of these spaces.

\begin{proposition}Let $X^n$ be a positively curved Alexandrov space of maximal symmetry rank. Then, if $n$ is odd, $X$ has finite cyclic fundamental group and $\chi (X) = 0$.

If $n$ is even, then $X$ is contractible if it has boundary, is simply connected if it is orientable or locally non-orientable, and has fundamental group of order two if it is locally orientable but not globally orientable.
If $X$ is closed and orientable, then
$\chi(X) = 2$ or $n+1$,  and $\chi(X) =1$ otherwise.\end{proposition}

\begin{proof}
The odd-dimensional spaces are spherical orbifolds, and we can see that the finite cyclic group $\Gamma_k$ given in the iterative description of the space above is the only group which might act without fixing points in the sphere, so by \cite{Ar} the result follows. The other cases are either trivial, or consequences of the Generalized Synge's Theorem \ref{t:GST}.

Because $X$ is homeomorphic to the orbit space of an isometric action on a Riemannian manifold, the fact that $\chi(X) = \chi(\Fix(X;T^1))$ for any $T^1$ in the torus now follows just as in the Riemannian case \cite{Ko}, noting that the subspaces to which the Lefschetz Fixed Point Theorem are applied are all triangulable by \cite{Jo}. The result is then true by induction.
\end{proof}

It is clear that a positively curved Riemannian orbifold is an example of a positively curved Alexandrov space.  Since all spaces in case (1) of Theorem D are spherical orbifolds, to classify the positively curved orbifolds with maximal symmetry rank, it suffices to understand when the spaces in case (2) of Theorem D produce orbifolds. Let $H \subset T^{k+1} \subset O(2k+2)$ be a connected rank one subgroup, and write $T^{k+1}$ as a product of circles $T_1 \times T_2 \times \cdots \times T_{k+1}$, each acting fixed-point homogeneously on the sphere. If $H=T_i$ for some $i$, then $X$ is not closed. If $H\subset T_{i_1} \times T_{i_2} \times \cdots \times T_{i_{l+1}}$ for pairwise different indices $i_j$ then $X$ is a join of some weighted complex projective space of complex dimension $l$ with a sphere. This is only homeomorphic to an orbifold if $l=k$, in which case $X$ is a weighted complex projective space, or if $l=1$, in which case $X$ is homeomorphic to a sphere, since $\ccc P^1$ is $S^2$. Hence we obtain the following corollary.

\begin{CorE}Let $X$ be an $n$-dimensional, closed
 Riemannian orbifold with positive sectional curvature admitting an isometric, (almost) effective $T^k$-action. Then $k\leq \lfloor \frac{n+1}{2} \rfloor$ and in the case of equality either
\begin{enumerate}
\item $X$ is a spherical orbifold, homeomorphic to $S^n / \Gamma$, where $\Gamma$ is a finite subgroup of the centralizer of the maximal torus in $O(n+1)$ or
\item only in the case that $n$ is even, $X$ is homeomorphic to a finite quotient of a weighted complex projective space $\ccc P^n_{a_0,a_1,\cdots,a_n}/\Gamma$, where $\Gamma$ is a finite subgroup of the linearly acting torus.
\end{enumerate}
\end{CorE}

\end{document}